\documentclass[letterpaper,reqno]{amsart}
\usepackage{amsfonts}
\usepackage{cite}

\newtheorem{theorem}{Theorem}
\theoremstyle{plain}

\newtheorem{corollary}{Corollary}

\newtheorem{definition}{Definition}

\newtheorem{lemma}{Lemma}
\newtheorem{notation}{Notation}

\newtheorem{proposition}{Proposition}
\newtheorem{remark}{Remark}

\numberwithin{equation}{section}

\begin{document}
\title[Sequential Warped Products: Curvature and Killing Vector Fields]{%
Sequential Warped Products: Curvature and conformal vector fields}
\date{July 30, 2016}

\begin{abstract}
In this note, we introduce a new type of warped products called as
sequential warped products to cover a wider variety of exact solutions to
Einstein's equation. First, we study the geometry of sequential warped
products and obtain covariant derivatives, curvature tensor, Ricci curvature
and scalar curvature formulas. Then some important consequences of these
formulas are also stated. We provide characterizations of geodesics and two
different types of conformal vector fields, namely, Killing vector fields
and concircular vector fields on sequential warped product manifolds.
Finally, we consider the geometry of two classes of sequential warped
product space-time models which are sequential generalized Robertson-Walker
spacetimes and sequential standard static spacetimes.
\end{abstract}

\author{Uday Chand De}
\address[U. C. De]{Department of Pure Mathematics, University of Calcutta,
35 Ballygaunge Circular Road, Kolkata 700019, West Bengala, India}
\email{uc$\_$de@yahoo.com}
\author{Sameh Shenawy}
\address[S. Shenawy]{Basic Science Department, Modern Academy for
Engineering and Technology, Maadi, Egypt}
\email{drssshenawy@eng.modern-academy.edu.eg, drshenawy@mail.com}
\author{B\"{u}lent \"{U}nal}
\address[B. \"{U}nal]{Department of Mathematics, Bilkent University,
Bilkent, 06800 Ankara, Turkey}
\email{bulentunal@mail.com}
\maketitle

\section{Introduction}

O'Neill and Bishop defined warped product manifolds to construct Riemannian
manifolds with negative sectional curvature\cite{Bishop:1969}. Since then
this notion has played some important roles in differential geometry as well
as in physics because warped product space-time models are used to obtain
exact solutions to Einstein's equation \cite{Apostolopoulos:2005, AD1, AD,
Berestovskii:2008, Besse2008, Ivancevic:2007, Oneill:1983}.

Doubly and multiply warped product manifolds are generalizations of (singly)
warped product manifolds\cite{Dobbaro2005, Unal2001A, Unal:2001}. In this
article, we define a new class of warped product manifolds, called as
sequential warped products where the base factor of the warped product is
itself a new warped product manifold. Sequential warped products can be
considered as a generalization of singly warped products. There are many
space-times where base, fiber or both are expressed as a warped product
manifolds. Among many such examples, we would like to mention especially
non-trivial ones such as Taub-Nut and stationary metrics (see \cite%
{Stephani:2003}) also Schwarzschild and generalized Riemannian anti de
Sitter $\mathbb{T}^2$ black hole metrics (see \S 3.2 of \cite{ACD} for
details). Moreover, some base conformal warped product space-times can be
expressed as a sequential warped product (see \cite{Dobbaro2008}).

We first introduce fundamental definitions about the new concept and state
some related remarks.

\begin{definition}
Let $M_{i}$ be three pseudo-Riemannian manifolds with metrics $g_{i}$ for $%
i=1,2,3$. Let $f:M_{1}\rightarrow (0,\infty )$ and $h:M_{1}\times
M_{2}\rightarrow (0,\infty )$ be two smooth positive functions on $M_{1}$
and $M_{1}\times M_{2}$, respectively. Then the sequential warped product
manifold, denoted by $\left( M_{1}\times _{f}M_{2}\right) \times _{h}M_{3}$,
is the triple product manifold $\bar{M} = \left( M_{1}\times M_{2}\right)
\times M_{3}$ furnished with the metric tensor%
\begin{equation*}
\bar{g}=\left( g_{1}\oplus f^{2}g_{2}\right) \oplus h^{2}g_{3}
\end{equation*}%
The functions $f$ and $h$ are called warping functions.
\end{definition}

Note that if $(M_{i},g_{i})$ are all Riemannian manifolds for any $i=1,2,3$,
then the sequential warped product manifold $\left( M_{1}\times
_{f}M_{2}\right) \times _{h}M_{3}$ is also a Riemannian manifold.

\begin{remark}
The warped product of the form $M_{1}\times _{f_{1}}\left( M_{2}\times
_{f_{2}}M_{3}\right) $ furnished by the metric%
\begin{equation*}
g=g_{1}+f_{1}^{2}\left( g_{2}+f_{2}^{2}g_{3}\right)
\end{equation*}%
is called the iterated warped product manifold of the manifolds $M_{1},$ $%
M_{2}$ and $M_{3}$. As a metric space, the iterated warped product manifold
is equal to the sequential warped product $\left( M_{1}\times
_{f}M_{2}\right) \times _{h}M_{3}$ where $f=f_{1}$ and $h=f_{1}f_{2}$.
Similarly, a sequential warped product $\left( M_{1}\times _{f}M_{2}\right)
\times _{h}M_{3}$ with a separable function $h:M_{1}\times M_{2}\rightarrow 
\mathbb{R}
$ is equal as a metric space to the iterated warped product manifold.
\end{remark}

\begin{remark}
If the warping function $h$ of the sequential warped product $\left(
M_{1}\times _{f}M_{2}\right) \times _{h}M_{3}$ is defined only on $M_{1}$,
then we have a multiply warped product manifold $M_{1}\times _{f}M_{2}\times
_{h}M_{3}$ with two fibers.
\end{remark}

\begin{remark}
A multiply warped product manifold of the form $M_{1}\times
_{f_{1}}M_{2}\times _{f_{2}}M_{3}$ is the sequential warped product manifold 
$\left( M_{1}\times _{f_{1}}M_{2}\right) \times _{f_{2}}M_{3}$ equipped with
the metric%
\begin{equation*}
g=\left( g_{1}+f_{1}^{2}g_{2}\right) +f_{2}^{2}g_{3}
\end{equation*}%
where both $f_{1}$ and $f_{2}$ are positive functions defined on $M_{1}$.
\end{remark}

Now, we would like to explain how to extend a generalized Robertson-Walker
space-time and a standard static space-time within the framework of
sequential warped products.

Let $(M_{i},g_{i})$ be two $n_{i}-$dimensional Riemannian manifolds for any $%
i=1,2$. Suppose that $I$ is an open, connected subinterval of $\mathbb{R}$
and $dt^{2}$ is the Euclidean metric tensor on $I$. Then

$\bullet $ An $(n_{1}+n_{2}+1)-$ dimensional product manifold $I\times
\left( M_{1}\times M_{2}\right) $ furnished with the metric tensor 
\begin{equation}
\bar{g}=-h^{2}dt^{2}\oplus \left( g_{1}\oplus f^{2}g_{2}\right)  \label{a}
\end{equation}%
is a sequential standard static space-time and is denoted by $\bar{M}%
=I_{h}\times \left( M_{1}\times _{f}M_{2}\right) $ where $h:M_{1}\times
M_{2}\rightarrow (0,\infty )$ and $f:M_{1}\rightarrow (0,\infty )$ are two
smooth functions.

Note that standard static space-times can be considered as a generalization
of the Einstein static universe\cite%
{AD1,AD,GES,Besse2008,Shenawy:2015,Shenawy:2016,Unal:2012}. Obviously, one
can obtain a standard static space-time from a sequential standard static
space-time by taking $M_{2}$ to be a singleton.

$\bullet $ An $(n_{1}+n_{2}+1)-$ dimensional product manifold $\left(
I\times M_{1}\right) \times M_{2}$ furnished with the metric tensor%
\begin{equation}
\bar{g}=-dt^{2}\oplus h^{2}\left( g_{1}\oplus f^{2}g_{2}\right) ,  \label{b}
\end{equation}%
is a sequential generalized Robertson-Walker space-time is denoted by $\bar{M%
}=I\times _{h}\left( M_{1}\times _{f}M_{2}\right) $ where $h:I\rightarrow
(0,\infty )$ and $f:M_{1}\rightarrow (0,\infty )$ are two smooth functions.

Note that generalized Robertson-Walker space-times can be considered as a
generalization of Robertson-Walker space-time \cite{Sanchez99, Sanchez98}.
As in the case of sequential standard static space-times, one can obtain a
generalized Robertson-Walker space-time from a sequential generalized
Robertson-Walker space-time by taking $M_{2}$ to be the empty set a
singleton.

In \cite{Stephani:2003}, there are many exact solutions of Einstein field
equation where the space-time may be written of the form $I\times \left(
M_{1}\times M_{2}\right) $ with metrics of the form (\ref{a}) or (\ref{b}).

Notice also that $\mathbb{S}_{1}^{n}\times F$ or $\mathbb{H}_{1}^{n}\times F$
are standard models in string theory where $F$ is a Calabi-Yau, Ricci flat
Riemannian Manifold and $\mathbb{S}_{1}^{n}$ is the de Sitter and also $%
\mathbb{H}_{1}^{n}$ is the anti-de Sitter manifold both of which are warped
product manifolds (see page 183 of \cite{BEE}). Thus sequential warped
product space-times play important role not only in the theory of general
relativity but also in the string theory.

In this article, we study some geometric concepts such as curvature,
geodesics, Killing vector fields and concircular vector fields on sequential
warped products. In section 2, we derive covariant derivative formulas for
sequential warped product manifolds. Then we derive many curvature formulas
such as Ricci curvature and scalar curvature formulas. In section 3, we
derive a characterization of two disjoint classes of conformal vector fields
on sequential warped product manifolds. In the last section, we apply our
results presented in Section 2 and Section 3, to sequential standard
space-times and generalized Robertson-Walker space-times.

Before we begin to state our main results, we would like to fix notations
used throughout the entire article.

\begin{notation}
Let $\bar{M}=\left( M_{1}\times _{f}M_{2}\right) \times _{h}M_{3}$ be a
sequential warped product manifold with metric $\bar g=\left( g_{1}\oplus
f^{2}g_{2}\right) \oplus h^{2}g_{3}$ where $f \colon M_1 \to (0, \infty)$
and $h \colon M_1 \times M_2 \to (0, \infty).$ Then

\begin{itemize}
\item $M=M_1 \times _{f}M_2$ is a warped product with the metric tensor $g =
g_{1}\oplus f^{2}g_{2}.$

\item $\mathrm{grad}^{1}f$ is the gradient of $f$ on $M_1$ and $\| \mathrm{%
grad}^{1}f \|^2 = g_1(\mathrm{grad}^{1}f, \mathrm{grad}^{1}f).$

\item $\mathrm{grad}h$ is the gradient of $h$ on $M$ and $\Vert \mathrm{grad}%
h\Vert ^{2}=g(\mathrm{grad}h,\mathrm{grad}h)$.

\item The same notation is used to denote a vector field and its lift to the
sequential warped product manifold.
\end{itemize}
\end{notation}

\section{Curvature of Sequential Warped Product Manifolds}

In this section, we will explore the geometry of sequential warped products
of the form $\left( M_{1}\times _{f}M_{2}\right) \times _{h}M_{3}$ by
providing the covariant derivative, curvature tensor, Ricci and scalar
curvature formulas. The proofs that are straightforward can be obtained by
applying similar results on singly warped products twice.

\begin{proposition}
\label{Connection}Let $\bar{M}=\left( M_{1}\times _{f}M_{2}\right) \times
_{h}M_{3}$ be a sequential warped product manifold with metric $\bar{g}%
=\left( g_{1}\oplus f^{2}g_{2}\right) \oplus h^{2}g_{3}$ and also let $%
X_{i},Y_{i}$ $\in \mathfrak{X}(M_{i})$ for any $i=1,2,3.$ Then

\begin{enumerate}
\item $\bar{\nabla}_{X_{1}}Y_{1}=\nabla _{X_{1}}^{1}Y_{1}$

\item $\bar{\nabla}_{X_{1}}X_{2}=\bar{\nabla}_{X_{2}}X_{1}=X_{1}\left( \ln
f\right) X_{2}$

\item $\bar{\nabla}_{X_{2}}Y_{2}=\nabla _{X_{2}}^{2}Y_{2}-fg_{2}\left(
X_{2},Y_{2}\right) \mathrm{grad}^{1}f$

\item $\bar{\nabla}_{X_{3}}X_{1}=\bar{\nabla}_{X_{1}}X_{3}=X_{1}\left( \ln
h\right) X_{3}$

\item $\bar{\nabla}_{X_{2}}X_{3}=\bar{\nabla}_{X_{3}}X_{2}=X_{2}\left( \ln
h\right) X_{3}$

\item $\bar{\nabla}_{X_{3}}Y_{3}=\nabla _{X_{3}}^{3}Y_{3}-hg_{3}\left(
X_{3},Y_{3}\right) \mathrm{grad}h$
\end{enumerate}
\end{proposition}

\begin{proposition}
\label{Curvature}Let $\bar{M}=\left( M_{1}\times _{f}M_{2}\right) \times
_{h}M_{3}$ be a sequential warped product manifold with metric $\bar{g}%
=\left( g_{1}\oplus f^{2}g_{2}\right) \oplus h^{2}g_{3}$ and also let $%
X_{i},Y_{i},Z_{i}$ $\in \mathfrak{X}(M_{i})$ for any $i=1,2,3.$ Then

\begin{enumerate}
\item \textrm{\={R}}$\left( X_{1},Y_{1}\right) Z_{1}=R^{1}\left(
X_{1},Y_{1}\right) Z_{1}$

\item \textrm{\={R}}$\left( X_{2},Y_{2}\right) Z_{2}=R^{2}\left(
X_{2},Y_{2}\right) Z_{2}-\left\Vert \mathrm{grad}^{1}f\right\Vert ^{2}\left[
g_{2}\left( X_{2},Z_{2}\right) Y_{2}-g_{2}\left( Y_{2},Z_{2}\right) X_{2}%
\right] $

\item \textrm{\={R}}$\left( X_{1},Y_{2}\right) Z_{1}=\dfrac{-1}{f}%
H_{1}^{f}\left( X_{1},Z_{1}\right) Y_{2}$

\item \textrm{\={R}}$\left( X_{1},Y_{2}\right) Z_{2}=fg_{2}\left(
Y_{2},Z_{2}\right) \nabla _{X_{1}}^{1}\mathrm{grad}^{1}f$

\item \textrm{\={R}}$\left( X_{1},Y_{2}\right) Z_{3}=0$

\item \textrm{\={R}}$\left( X_{i},Y_{i}\right) Z_{j}=0,i\neq j$

\item \textrm{\={R}}$\left( X_{i},Y_{3}\right) Z_{j}=\dfrac{-1}{h}%
H^{h}\left( X_{i},Z_{j}\right) Y_{3},i,j=1,2$

\item \textrm{\={R}}$\left( X_{i},Y_{3}\right) Z_{3}=hg_{3}\left(
Y_{3},Z_{3}\right) \nabla _{X_{i}}\mathrm{grad}h,i=1,2$

\item \textrm{\={R}}$\left( X_{3},Y_{3}\right) Z_{3}=R^{3}\left(
X_{3},Y_{3}\right) Z_{3}-\left\Vert \mathrm{grad}h\right\Vert ^{2}\left[
g_{3}\left( X_{3},Z_{3}\right) Y_{3}-g_{3}\left( Y_{3},Z_{3}\right) X_{3}%
\right] $
\end{enumerate}
\end{proposition}

Now consider the Ricci curvature denoted by \textrm{\={R}ic} of a sequential
warped product of the form $\left( M_{1}\times _{f}M_{2}\right) \times
_{h}M_{3}$.

\begin{proposition}
\label{Ricci}Let $\bar{M}=\left( M_{1}\times _{f}M_{2}\right) \times
_{h}M_{3}$ be a sequential warped product manifold with metric $\bar{g}%
=\left( g_{1}\oplus f^{2}g_{2}\right) \oplus h^{2}g_{3}$ and also let $%
X_{i},Y_{i},Z_{i}$ $\in \mathfrak{X}(M_{i})$ for any $i=1,2,3.$ Then

\begin{enumerate}
\item \textrm{\={R}ic}$\left( X_{1},Y_{1}\right) =$\textrm{Ric}$^{1}\left(
X_{1},Y_{1}\right) -\dfrac{n_{2}}{f}H_{1}^{f}\left( X_{1},Y_{1}\right) -%
\dfrac{n_{3}}{h}H^{h}\left( X_{1},Y_{1}\right) $

\item \textrm{\={R}ic}$\left( X_{2},Y_{2}\right) =$\textrm{Ric}$^{2}\left(
X_{2},Y_{2}\right) -f^{\sharp }g_{2}\left( X_{2},Y_{2}\right) -\dfrac{n_{3}}{%
h}H^{h}\left( X_{2},Y_{2}\right) $

\item \textrm{\={R}ic}$\left( X_{3},Y_{3}\right) =$\textrm{Ric}$^{3}\left(
X_{3},Y_{3}\right) -h^{\sharp }g_{3}\left( X_{3},Y_{3}\right) $

\item \textrm{\={R}ic}$\left( X_{i},Y_{j}\right) =0,i\neq j$

where $f^{\sharp }=f\Delta ^{1}f+\left( n_{2}-1\right) \left\Vert \mathrm{%
grad}^{1}f\right\Vert ^{2}$ and $h^{\sharp }=h\Delta h+\left( n_{3}-1\right)
\left\Vert \mathrm{grad}h\right\Vert ^{2}$
\end{enumerate}
\end{proposition}

We now apply the last result to establish conditions for a sequential warped
product to be Einstein.

\begin{theorem}
The sequential warped product $\left( M_{1}\times _{f}M_{2}\right) \times
_{h}M_{3}$\ is Einstein with \textrm{\={R}ic}$=\lambda \bar{g}$ if and only
if

\begin{enumerate}
\item \textrm{Ric}$^{1}=\lambda g_{1}+\dfrac{n_{2}}{f}H_{1}^{f}+\dfrac{n_{3}%
}{h}H^{h}$

\item \textrm{Ric}$^{2}=\left( \lambda f^{2}+f^{\sharp}\right) g_{2}+\dfrac{%
n_{3}}{h}H^{h}$

\item $M_{3}$ is Einstein with \textrm{Ric}$^{3}=\left( \lambda
h^{2}+h^{\sharp}\right) g_{3}$.
\end{enumerate}
\end{theorem}

In \cite{Dobarro:1987}, F. Dobarro and E. Lam\'{\i} Dozo established a
relationship between the scalar curvature of a warped product of the form $%
M\times _{f}N$ and that of its base and fiber manifolds $M$ and $N$. In the
following theorem we derive a quite different result for a sequential warped
product manifold.

\begin{theorem}
Let $\bar{M}=\left( M_{1}\times _{f}M_{2}\right) \times _{h}M_{3}$ be a
sequential warped product manifold with metric $\bar{g}=\left( g_{1}\oplus
f^{2}g_{2}\right) \oplus h^{2}g_{3}$ and let $r_{i}$ be the scalar curvature
of $M_{i}$, $i=1,2,3$. Then the scalar curvature $r$ of $\bar{M}$ is given by%
\begin{equation*}
r=r_{1}+\frac{r_{2}}{f^{2}}+\frac{r_{3}}{h^{2}}-\frac{2n_{2}}{f}\Delta ^{1}f-%
\dfrac{2n_{3}}{h}\Delta h-\frac{n_{2}\left( n_{2}-1\right) }{f^{2}}%
\left\Vert \mathrm{grad}^{1}f\right\Vert ^{2}-\frac{n_{3}\left(
n_{3}-1\right) }{h^{2}}\left\Vert \mathrm{grad}h\right\Vert ^{2}
\end{equation*}
\end{theorem}

\begin{proof}
Let $\left\{ e_{1},e_{2},...,e_{n_{1}}\right\} $, $\left\{
e_{n_{1}+1},e_{n_{1}+2},...,e_{n_{1}+n_{2}}\right\} $ and $\left\{
e_{n_{1}+n_{2}+1}+e_{n_{1}+n_{2}+2},...,e_{n}\right\} $ be three frames over 
$M_{1}$, $M_{2}$ and $M_{3}$ respectively. The scalar curvature $r$ of $\bar{%
M}$ is given by%
\begin{eqnarray*}
r &=&\sum_{i=1}^{n_{1}}\mathrm{\bar{R}ic}\left( e_{i},e_{i}\right) +\frac{1}{%
f^{2}}\sum_{i=n_{1}+1}^{n_{1}+n_{2}}\mathrm{\bar{R}ic}\left(
e_{i},e_{i}\right) +\frac{1}{h^{2}}\sum_{i=n_{1}+n_{2}+1}^{n_{1}+n_{2}+n_{3}}%
\mathrm{\bar{R}ic}\left( e_{i},e_{i}\right) \\
&=&r_{1}-\frac{n_{2}}{f}\Delta ^{1}f-\dfrac{n_{3}}{h}\sum_{i=1}^{n_{1}}H^{h}%
\left( e_{i},e_{i}\right) +\frac{1}{f^{2}}r_{2}-\frac{n_{2}}{f^{2}}f^{\sharp
}-\frac{1}{f^{2}}\frac{n_{3}}{h}\sum_{i=n_{1}+1}^{n_{1}+n_{2}}H^{h}\left(
e_{i},e_{i}\right) \\
&&+\frac{1}{h^{2}}\left[ r_{3}-h^{\sharp }n_{3}\right] \\
&=&r_{1}+\frac{1}{f^{2}}r_{2}+\frac{1}{h^{2}}r_{3}-\frac{n_{2}}{f}\Delta
^{1}f-\dfrac{n_{3}}{h}\Delta h-\frac{n_{2}}{f^{2}}f^{\sharp}-\frac{n_{3}}{%
h^{2}}h^{\sharp} \\
&=&r_{1}+\frac{r_{2}}{f^{2}}+\frac{r_{3}}{h^{2}}-\frac{2n_{2}}{f}\Delta
^{1}f-\dfrac{2n_{3}}{h}\Delta h-\frac{n_{2}\left( n_{2}-1\right) }{f^{2}}%
\left\Vert \mathrm{grad}^{1}f\right\Vert ^{2}-\frac{n_{3}\left(
n_{3}-1\right) }{h^{2}}\left\Vert \mathrm{grad}h\right\Vert ^{2}
\end{eqnarray*}
\end{proof}

Suppose that $\bar{M}=\left( M_{1}\times _{f}M_{2}\right) \times _{h}M_{3}$
has a constant sectional curvature $\kappa $. Then the first item of
Proposition (\ref{Curvature}) yields%
\begin{eqnarray*}
\mathrm{\bar{R}}\left( X_{1},Y_{1}\right) Z_{1} &=&\kappa \left\{
g_{1}\left( X_{1},Z_{1}\right) Y_{1}-g_{1}\left( Y_{1},Z_{1}\right)
X_{1}\right\} \\
\mathrm{\bar{R}}\left( X_{1},Y_{1}\right) Z_{1} &=&R^{1}\left(
X_{1},Y_{1}\right) Z_{1}
\end{eqnarray*}%
Thus $M_{1}$ has a constant sectional curvature $\kappa _{1}=\kappa $. The
second item implies that%
\begin{eqnarray*}
\mathrm{\bar{R}}\left( X_{2},Y_{2}\right) Z_{2} &=&\kappa \left\{ g\left(
X_{2},Z_{2}\right) Y_{2}-g\left( Y_{2},Z_{2}\right) X_{2}\right\} \\
&=&\kappa f^{2}\left\{ g_{2}\left( X_{2},Z_{2}\right) Y_{2}-g_{2}\left(
Y_{2},Z_{2}\right) X_{2}\right\} \\
\mathrm{\bar{R}}\left( X_{2},Y_{2}\right) Z_{2} &=&R^{2}\left(
X_{2},Y_{2}\right) Z_{2}-\left\Vert \mathrm{grad}^{1}f\right\Vert
^{2}\left\{ g_{2}\left( X_{2},Z_{2}\right) Y_{2}-g_{2}\left(
Y_{2},Z_{2}\right) X_{2}\right\}
\end{eqnarray*}%
Therefore, Shur's Lemma implies that $M_{2}$ has a constant sectional
curvature $\kappa_2$ given by%
\begin{equation*}
\kappa _{2}=\kappa f^{2}+\left\Vert \mathrm{grad}^{1}f\right\Vert ^{2}
\end{equation*}%
for $n_{2}\geq 3$. Similarly, $M_{3}$ has a constant sectional curvature
curvature $\kappa_3$ given by%
\begin{equation*}
\kappa _{3}=\kappa h^{2}+\left\Vert \mathrm{grad}h\right\Vert ^{2}
\end{equation*}%
for $n_{3}\geq 3$.

\begin{theorem}
Let $\bar{M}=\left( M_{1}\times _{f}M_{2}\right) \times _{h}M_{3}$ be a
sequential warped product manifold with metric $\bar{g}=\left( g_{1}\oplus
f^{2}g_{2}\right) \oplus h^{2}g_{3}$ and let $X_{i},Y_{i},Z_{i}$ $\in 
\mathfrak{X}(M_{i})$ for any $i=1,2,3$. Assume that $\bar{M}$ has a constant
sectional curvature $\kappa $. Then

\begin{enumerate}
\item $M_{1}$ has a costant sectional curvature $\kappa _{1}=\kappa $,

\item $M_{2}$ has a costant sectional curvature $\kappa _{2}=\kappa
f^{2}+\left\Vert \mathrm{grad}^{1}f\right\Vert ^{2}$ for $n_{2}\geq 3$, and

\item $M_{3}$ has a costant sectional curvature $\kappa _{3}=\kappa
h^{2}+\left\Vert \mathrm{grad}h\right\Vert ^{2}$ for $n_{3}\geq 3$.
\end{enumerate}
\end{theorem}

\section{Conformal vector fields}

Conformal vector fields have well-known geometrical and physical
interpretations and have been studied for a long time by geometers and
physicists on Riemannian and pseudo-Riemannian manifolds. Killing vector
fields are conformal vector fields on (pseudo-) Riemannian manifolds that
preserve metric, i.e, under the flow of a Killing vector field the metric
does not change. The set of all Killing vector fields on a connected
Riemannian manifold forms a Lie algebra over the set of real numbers.\cite%
{Berestovskii:2008, Besse2008, Duggala:2005, Ivancevic:2007, Kuhnel:1997,
Operea:2008, Shenawy:2015}

In \cite{Unal:2012}, the authors studied Killing vector fields of warped
product manifolds specially on standard static space-times. They prove some
global characterization of the Killing vector fields of a standard static
space-time. More explicitly, they obtain a form of a Killing vector field on
this class of space-times. Moreover, a characterization of the Killing
vector fields on a standard static space-time with compact Riemannian parts
and many other interesting results are given. In this section, we study the
concept of conformal vector fields on sequential warped product manifolds.

A vector field $\zeta $on a Riemannian manifold $\left( M,g\right) $ is
conformal if%
\begin{equation}
\mathcal{L}_{\zeta }g=\rho g
\end{equation}%
where $\mathcal{L}_{\zeta }$ is the Lie derivative in direction of the
vector field $\zeta $. Moreover, $\zeta $ is called a Killing vector field
if $\rho =0.$ This is equivalent to say that $\zeta $ is Killing if%
\begin{equation}
g\left( \nabla _{X}\zeta ,Y\right) +g\left( X,\nabla _{Y}\zeta \right) =0
\end{equation}%
for any vector fields $X,Y\in \mathfrak{X}\left( M\right) $. By symmetry of
the above equation, $\zeta $ is Killing if%
\begin{equation}
g\left( \nabla _{X}\zeta ,X\right) =0  \label{e1}
\end{equation}%
for any vector field $X\in \mathfrak{X}\left( M\right) $.

From now on $\bar{M}=\left( M_{1}\times _{f}M_{2}\right) \times _{h}M_{3}$
denotes a sequential warped product manifold with metric $\bar g=\left(
g_{1}\oplus f^{2}g_{2}\right) \oplus h^{2}g_{3}.$

\begin{theorem}
A vector field $\zeta \in \mathfrak{X}\left( \left( M_{1}\times
_{f}M_{2}\right) \times _{h}M_{3}\right) $ is Killing if

\begin{enumerate}
\item $\zeta _{i}$ is Killing on $M_{i},$ for every $i=1,2,3$

\item $\zeta _{1}\left( f\right) =0$

\item $\left( \zeta _{1}+\zeta _{2}\right) h=0$
\end{enumerate}
\end{theorem}

\begin{proof}
The vector field $\zeta \in \mathfrak{X}\left( \bar{M}\right) $ is Killing
by equation (\ref{e1}) if and only if%
\begin{equation*}
\bar{g}\left( \bar{\nabla}_{X}\zeta ,X\right) =0
\end{equation*}%
for any vector field $X\in \mathfrak{X}\left( \bar{M}\right) $. It is clear
that%
\begin{eqnarray*}
\bar{g}\left( \bar{\nabla}_{X}\zeta ,X\right) &=&\bar{g}\left( \bar{\nabla}%
_{X_{1}}\zeta _{1}+\bar{\nabla}_{X_{1}}\zeta _{2}+\bar{\nabla}_{X_{1}}\zeta
_{3},X\right) \\
&&+\bar{g}\left( \bar{\nabla}_{X_{2}}\zeta _{1}+\bar{\nabla}_{X_{2}}\zeta
_{2}+\bar{\nabla}_{X_{2}}\zeta _{3},X\right) \\
&&+\bar{g}\left( \bar{\nabla}_{X_{3}}\zeta _{1}+\bar{\nabla}_{X_{3}}\zeta
_{2}+\bar{\nabla}_{X_{3}}\zeta _{3},X\right)
\end{eqnarray*}%
Now using Proposition (\ref{Connection}) we have%
\begin{eqnarray*}
&&\bar{g}\left( \bar{\nabla}_{X}\zeta ,X\right) \\
&=&\bar{g}\left( \nabla _{X_{1}}^{1}\zeta _{1}+X_{1}\left( \ln f\right)
\zeta _{2}+X_{1}\left( \ln h\right) \zeta _{3},X\right) \\
&&+\bar{g}\left( \zeta _{1}\left( \ln f\right) X_{2}+\nabla
_{X_{2}}^{2}\zeta _{2}-fg_{2}\left( \zeta _{2},X_{2}\right) \mathrm{grad}%
^{1}f+X_{2}\left( \ln h\right) \zeta _{3},X\right) \\
&&+\bar{g}\left( \zeta _{1}\left( \ln h\right) X_{3}+\zeta _{2}\left( \ln
h\right) X_{3}+\nabla _{X_{3}}^{3}\zeta _{3}-hg_{3}\left( \zeta
_{3},X_{3}\right) \mathrm{grad}h,X\right) \\
&=&g_{1}\left( \nabla _{X_{1}}^{1}\zeta _{1},X_{1}\right) +f^{2}g_{2}\left(
\nabla _{X_{2}}^{2}\zeta _{2},X_{2}\right) +h^{2}g_{3}\left( \nabla
_{X_{3}}^{3}\zeta _{3},X_{3}\right) \\
&&+f\zeta _{1}\left( f\right) g_{2}\left( X_{2},X_{2}\right) +h\left( \zeta
_{1}+\zeta _{2}\right) \left( h\right) g_{3}\left( X_{3},X_{3}\right)
\end{eqnarray*}%
From this equation one can easily deduce the result.
\end{proof}

The following result will enable us to discuss the converse of the above
result.

\begin{proposition}
A vector field $\zeta \in \mathfrak{X}\left( \left( M_{1}\times
_{f}M_{2}\right) \times _{h}M_{3}\right) $ satisfies%
\begin{eqnarray}
\left( \mathcal{L}_{\zeta }g\right) \left( X,Y\right) &=&\left( \mathcal{L}%
_{\zeta _{1}}^{1}g_{1}\right) \left( X_{1},Y_{1}\right) +f^{2}\left( 
\mathcal{L}_{\zeta _{2}}^{2}g_{2}\right) \left( X_{2},Y_{2}\right)
+h^{2}\left( \mathcal{L}_{\zeta _{3}}^{3}g_{3}\right) \left(
X_{3},Y_{3}\right)  \notag \\
&&+2f\zeta _{1}\left( f\right) g_{2}\left( X_{2},Y_{2}\right) +2h\left(
\zeta _{1}+\zeta _{2}\right) \left( h\right) g_{3}\left( X_{3},Y_{3}\right)
\label{e4}
\end{eqnarray}%
for any vector fields $X,Y\in \mathfrak{X}\left( \left( M_{1}\times
_{f}M_{2}\right) \times _{h}M_{3}\right) $.
\end{proposition}

\begin{theorem}
Let $\zeta \in \mathfrak{X}\left( \left( M_{1}\times _{f}M_{2}\right) \times
_{h}M_{3}\right) $ be a Killing vector field. Then

\begin{enumerate}
\item $\zeta _{1}$ is Killing on $M_{1}$,

\item $\zeta _{2}$ is conformal on $M_{2}$ with conformal factor $-2\zeta
_{1}\left( \ln f\right) $,

\item $\zeta _{3}$ is conformal on $M_{3}$ with conformal factor $-2\left(
\zeta _{1}+\zeta _{2}\right) \left( \ln h\right) $.
\end{enumerate}
\end{theorem}

\begin{proof}
Consider equation (\ref{e4}). We have the following cases. By substituting $%
X=X_{1}$ and $Y=Y_{1},$ we obtain%
\begin{equation*}
\left( \mathcal{L}_{\zeta _{1}}^{1}g_{1}\right) \left( X_{1},Y_{1}\right) =0
\end{equation*}%
and thus $\zeta _{1}$ is Killing. Now, let $X=X_{2}$ and $Y=Y_{2},$ be then
we have 
\begin{eqnarray*}
0 &=&f^{2}\left( \mathcal{L}_{\zeta _{2}}^{2}g_{2}\right) \left(
X_{2},Y_{2}\right) +2f\zeta _{1}\left( f\right) g_{2}\left(
X_{2},Y_{2}\right) \\
\left( \mathcal{L}_{\zeta _{2}}^{2}g_{2}\right) \left( X_{2},Y_{2}\right)
&=&-2\zeta _{1}\left( \ln f\right) g_{2}\left( X_{2},Y_{2}\right)
\end{eqnarray*}%
and thus $\zeta _{2}$ is conformal. Finally, if $X=X_{3}$ and $Y=Y_{3}$, then%
\begin{eqnarray*}
0 &=&h^{2}\left( \mathcal{L}_{\zeta _{3}}^{3}g_{3}\right) \left(
X_{3},Y_{3}\right) +2h\left( \zeta _{1}+\zeta _{2}\right) \left( h\right)
g_{3}\left( X_{3},Y_{3}\right) \\
\left( \mathcal{L}_{\zeta _{3}}^{3}g_{3}\right) \left( X_{3},Y_{3}\right)
&=&-2\left( \zeta _{1}+\zeta _{2}\right) \left( \ln h\right) g_{3}\left(
X_{3},Y_{3}\right)
\end{eqnarray*}%
and thus $\zeta _{3}$ is conformal.
\end{proof}

\begin{theorem}
Let $\zeta \in \mathfrak{X}\left( \left( M_{1}\times _{f}M_{2}\right) \times
_{h}M_{3}\right) $ be a vector field on a sequential warped product
manifold. Assume that

\begin{enumerate}
\item $\zeta _{i}$ is conformal on $M_{i}$ with factor $\rho _{i}$ for each $%
i$,

\item $\rho _{1}=\rho _{2}+2\zeta _{1}\left( \ln f\right) $,

\item $\rho _{1}=\rho _{3}+2\left( \zeta _{1}+\zeta _{2}\right) \left( \ln
h\right) $.
\end{enumerate}

Then $\zeta $ is conformal on $\bar{M}$.
\end{theorem}

Now, we will study the geodesic curves and their equations on a sequential
warped product. In a sequential warped product of the form $\left(
M_{1}\times _{f}M_{2}\right) \times _{h}M_{3}$, as product manifold, a curve 
$\alpha \left( t\right) $ can be written as $\alpha \left( t\right) =\left(
\alpha _{1}\left( t\right) ,\alpha _{2}\left( t\right) ,\alpha _{3}\left(
t\right) \right) $ with $\alpha _{i}\left( t\right) $ the projections of $%
\alpha $ into $M_{i}$ for any $i=1,2,3$.

\begin{lemma}
\label{geodesic}Let $\alpha \left( t\right) =\left( \alpha _{1}\left(
t\right) ,\alpha _{2}\left( t\right) ,\alpha _{3}\left( t\right) \right) $
be a smooth curve on a sequential warped product of the form $\bar{M}=\left(
M_{1}\times _{f}M_{2}\right) \times _{h}M_{3}$ with metric $\bar{g}=\left(
g_{1}\oplus f^{2}g_{2}\right) \oplus h^{2}g_{3}$. Then $\alpha $ is a
geodesic in $\bar{M}$ if and only if

\begin{enumerate}
\item $\nabla _{\dot{\alpha}_{1}}^{1}\dot{\alpha}_{1}=f\left\Vert \dot{\alpha%
}_{2}\right\Vert _{2}^{2}\mathrm{grad}^{1}f+h\left\Vert \dot{\alpha}%
_{3}\right\Vert _{3}^{2}\left( \mathrm{grad}h\right) ^{T}$ on $M_{1}$

\item $\nabla _{\dot{\alpha}_{2}}^{2}\dot{\alpha}_{2}=-2\dot{\alpha}%
_{1}\left( \ln f\right) \dot{\alpha}_{2}+h\left\Vert \dot{\alpha}%
_{3}\right\Vert _{3}^{2}\left( \mathrm{grad}h\right) ^{\perp }$ on $M_{2}$

\item $\nabla _{\dot{\alpha}_{3}}^{3}\dot{\alpha}_{3}=-2\dot{\alpha}%
_{1}\left( \ln h\right) \dot{\alpha}_{3}-2\dot{\alpha}_{2}\left( \ln
h\right) \dot{\alpha}_{3}$ on $M_{3}$
\end{enumerate}
\end{lemma}

\begin{proof}
Then $\alpha _{i}\left( t\right) $ is regular hence we can suppose $\alpha
_{i}\left( t\right) $ is an integral curve of $\dot{\alpha}_{i}$ on $M_{i}$
and so $\alpha \left( t\right) $ is an integral curve of $\dot{\alpha}=\dot{%
\alpha}_{1}+\dot{\alpha}_{2}+\dot{\alpha}_{3}$. Thus%
\begin{eqnarray*}
\bar{\nabla}_{\dot{\alpha}}\dot{\alpha} &=&\bar{\nabla}_{\dot{\alpha}_{1}}%
\dot{\alpha}_{1}+\bar{\nabla}_{\dot{\alpha}_{1}}\dot{\alpha}_{2}+\bar{\nabla}%
_{\dot{\alpha}_{1}}\dot{\alpha}_{3} \\
&&+\bar{\nabla}_{\dot{\alpha}_{2}}\dot{\alpha}_{1}+\bar{\nabla}_{\dot{\alpha}%
_{2}}\dot{\alpha}_{2}+\bar{\nabla}_{\dot{\alpha}_{2}}\dot{\alpha}_{3} \\
&&+\bar{\nabla}_{\dot{\alpha}_{3}}\dot{\alpha}_{1}+\bar{\nabla}_{\dot{\alpha}%
_{3}}\dot{\alpha}_{2}+\bar{\nabla}_{\dot{\alpha}_{3}}\dot{\alpha}_{3}
\end{eqnarray*}%
Now we apply Proposition (\ref{Connection}) to get%
\begin{eqnarray*}
\bar{\nabla}_{\dot{\alpha}}\dot{\alpha} &=&\nabla _{\dot{\alpha}_{1}}^{1}%
\dot{\alpha}_{1}+2\dot{\alpha}_{1}\left( \ln f\right) \dot{\alpha}_{2}+2\dot{%
\alpha}_{1}\left( \ln h\right) \dot{\alpha}_{3} \\
&&+2\dot{\alpha}_{2}\left( \ln h\right) \dot{\alpha}_{3}+\nabla _{\dot{\alpha%
}_{2}}^{2}\dot{\alpha}_{2}-fg_{2}\left( \dot{\alpha}_{2},\dot{\alpha}%
_{2}\right) \mathrm{grad}^{1}f \\
&&+\nabla _{\dot{\alpha}_{3}}^{3}\dot{\alpha}_{3}-hg_{3}\left( \dot{\alpha}%
_{3},\dot{\alpha}_{3}\right) \mathrm{grad}h
\end{eqnarray*}%
This equation implies that%
\begin{eqnarray*}
\bar{\nabla}_{\dot{\alpha}}\dot{\alpha} &=&\nabla _{\dot{\alpha}_{1}}^{1}%
\dot{\alpha}_{1}-fg_{2}\left( \dot{\alpha}_{2},\dot{\alpha}_{2}\right) 
\mathrm{grad}^{1}f-hg_{3}\left( \dot{\alpha}_{3},\dot{\alpha}_{3}\right)
\left( \mathrm{grad}h\right) ^{T} \\
&&+\nabla _{\dot{\alpha}_{2}}^{2}\dot{\alpha}_{2}+2\dot{\alpha}_{1}\left(
\ln f\right) \dot{\alpha}_{2}-hg_{3}\left( \dot{\alpha}_{3},\dot{\alpha}%
_{3}\right) \left( \mathrm{grad}h\right) ^{\perp } \\
&&+\nabla _{\dot{\alpha}_{3}}^{3}\dot{\alpha}_{3}+2\dot{\alpha}_{1}\left(
\ln h\right) \dot{\alpha}_{3}+2\dot{\alpha}_{2}\left( \ln h\right) \dot{%
\alpha}_{3}
\end{eqnarray*}
\end{proof}

\begin{theorem}
Let $\zeta \in \mathfrak{X}\left( \left( M_{1}\times _{f}M_{2}\right) \times
_{h}M_{3}\right) $ be a Killing vector field. Then $g\left( \zeta ,X\right) $
is constant along the integral curve $\alpha \left( t\right) =\left( \alpha
_{1}\left( t\right) ,\alpha _{2}\left( t\right) ,\alpha _{3}\left( t\right)
\right) $ of $X$ if

\begin{enumerate}
\item $\nabla _{X_{1}}^{1}X_{1}=f\left\Vert \dot{\alpha}_{2}\right\Vert
_{2}^{2}\mathrm{grad}^{1}f+h\left\Vert \dot{\alpha}_{3}\right\Vert
_{3}^{2}\left( \mathrm{grad}h\right) ^{T}$ on $M_{1}$

\item $\nabla _{X_{2}}^{2}X_{2}=-2X_{1}\left( \ln f\right) X_{2}+h\left\Vert 
\dot{\alpha}_{3}\right\Vert _{3}^{2}\left( \mathrm{grad}h\right) ^{\perp }$
on $M_{2}$

\item $\nabla _{X_{3}}^{3}X_{3}=-2X_{1}\left( \ln h\right)
X_{3}-2X_{2}\left( \ln h\right) X_{3}$ on $M_{3}.$
\end{enumerate}
\end{theorem}

\begin{proof}
The conditions (1-3) imply that $\alpha \left( t\right) $ is a geodesic and
so $\nabla _{X}X=0$ (see Lemma \ref{geodesic}). Thus $g\left( \zeta
,X\right) $ is constant along the integral curve of $X$.
\end{proof}

A vector field $\zeta $ on a Riemannian manifold $M$ is called concircular
vector field if%
\begin{equation*}
\nabla _{X}\zeta =\mu X
\end{equation*}%
for any vector field $X$ where $\mu $ is function defined on $M$. It is
clear that%
\begin{equation*}
\left( \mathcal{L}_{\zeta }g\right) \left( X,Y\right) =2\mu g\left(
X,Y\right)
\end{equation*}%
i.e. any concircular vector field is a conformal vector field. Concircular
vector fields have many applications in geometry and physics\cite{Chen2014}.
A concircular vector field is sometimes called a closed conformal vector
field.

\begin{theorem}
Let $\zeta \in \mathfrak{X}\left( \left( M_{1}\times _{f}M_{2}\right) \times
_{h}M_{3}\right) $ be a concircular vector field on $\bar{M}=\left(
M_{1}\times _{f}M_{2}\right) \times _{h}M_{3}$. Then each $\zeta _{i}$ is a
non-zero concircular vector field on $M_{i}$ for any $i=1,2,3$ if and only
if both $f$ and $h$ are constant functions.
\end{theorem}

\begin{proof}
Using the definition of concircular vector fields and Theorem \ref%
{Connection}, we obtain that%
\begin{eqnarray*}
\nabla _{X}\zeta &=&\nabla _{X_{1}}\zeta _{1}+\nabla _{X_{1}}\zeta
_{2}+\nabla _{X_{1}}\zeta _{3}+\nabla _{X_{2}}\zeta _{1}+\nabla
_{X_{2}}\zeta _{2}+\nabla _{X_{2}}\zeta _{3}+\nabla _{X_{3}}\zeta
_{1}+\nabla _{X_{3}}\zeta _{2}+\nabla _{X_{3}}\zeta _{3} \\
\mu X &=&\nabla _{X_{1}}^{1}\zeta _{1}+X_{1}\left( f\right) \zeta
_{2}+X_{1}\left( h\right) \zeta _{3}+\zeta _{1}\left( f\right) X_{2}+\nabla
_{X_{2}}^{2}\zeta _{2}-fg_{2}\left( X_{2},\zeta _{2}\right) \mathrm{grad}%
^{1}f \\
&&+X_{2}\left( h\right) \zeta _{3}+\zeta _{1}\left( h\right) X_{3}+\zeta
_{2}\left( h\right) X_{3}+\nabla _{X_{3}}^{3}\zeta _{3}-hg_{3}\left(
X_{3},\zeta _{3}\right) \mathrm{grad}h
\end{eqnarray*}%
Suppose that both $f$ and $h$ are constant functions, then 
\begin{eqnarray}
\nabla _{X_{1}}^{1}\zeta _{1}-fg_{2}\left( X_{2},\zeta _{2}\right) \mathrm{%
grad}^{1}f-hg_{3}\left( X_{3},\zeta _{3}\right) \left( \mathrm{grad}h\right)
^{T} &=&\mu X_{1}  \notag \\
\nabla _{X_{2}}^{2}\zeta _{2}+X_{1}\left( f\right) \zeta _{2}+\zeta
_{1}\left( f\right) X_{2}-hg_{3}\left( X_{3},\zeta _{3}\right) \left( 
\mathrm{grad}h\right) ^{\perp } &=&\mu X_{2}  \label{a1} \\
\nabla _{X_{3}}^{3}\zeta _{3}+X_{1}\left( h\right) \zeta _{3}+X_{2}\left(
h\right) \zeta _{3}+\zeta _{1}\left( h\right) X_{3}+\zeta _{2}\left(
h\right) X_{3} &=&\mu X_{3}  \notag
\end{eqnarray}%
Now, suppose that both $f$ and $h$ are constant functions, then%
\begin{eqnarray*}
\nabla _{X_{1}}^{1}\zeta _{1} &=&\mu X_{1} \\
\nabla _{X_{2}}^{2}\zeta _{2} &=&\mu X_{2} \\
\nabla _{X_{3}}^{3}\zeta _{3} &=&\mu X_{3}
\end{eqnarray*}%
i.e., each $\zeta _{i}$ is concircular on $M_{i}$ for $i=1,2,3$. Conversely,
we suppose that 
\begin{eqnarray*}
\nabla _{X_{1}}^{1}\zeta _{1} &=&\mu _{1}X_{1} \\
\nabla _{X_{2}}^{2}\zeta _{2} &=&\mu _{2}X_{2} \\
\nabla _{X_{3}}^{3}\zeta _{3} &=&\mu _{3}X_{3}
\end{eqnarray*}

Hence Equation \ref{a1} becomes%
\begin{eqnarray*}
\mu _{1}X_{1}-fg_{2}\left( X_{2},\zeta _{2}\right) \mathrm{grad}%
^{1}f-hg_{3}\left( X_{3},\zeta _{3}\right) \left( \mathrm{grad}h\right) ^{T}
&=&\mu X_{1} \\
\mu _{2}X_{2}+X_{1}\left( f\right) \zeta _{2}+\zeta _{1}\left( f\right)
X_{2}-hg_{3}\left( X_{3},\zeta _{3}\right) \left( \mathrm{grad}h\right)
^{\perp } &=&\mu X_{2} \\
\mu _{3}X_{3}+X_{1}\left( h\right) \zeta _{3}+X_{2}\left( h\right) \zeta
_{3}+\zeta _{1}\left( h\right) X_{3}+\zeta _{2}\left( h\right) X_{3} &=&\mu
X_{3}
\end{eqnarray*}%
\begin{eqnarray}
\bar{\mu}_{1}X_{1}-fg_{2}\left( X_{2},\zeta _{2}\right) \mathrm{grad}%
^{1}f-hg_{3}\left( X_{3},\zeta _{3}\right) \left( \mathrm{grad}h\right) ^{T}
&=&0  \label{b1} \\
\bar{\mu}_{2}X_{2}+X_{1}\left( f\right) \zeta _{2}+\zeta _{1}\left( f\right)
X_{2}-hg_{3}\left( X_{3},\zeta _{3}\right) \left( \mathrm{grad}h\right)
^{\perp } &=&0  \label{b2} \\
\bar{\mu}_{3}X_{3}+X_{1}\left( h\right) \zeta _{3}+X_{2}\left( h\right)
\zeta _{3}+\zeta _{1}\left( h\right) X_{3}+\zeta _{2}\left( h\right) X_{3}
&=&0  \label{b3}
\end{eqnarray}%
These equations must be satisfied by any arbitrary vector field $X$. Let us
put $X_{3}=0$ in Equation \ref{b3}, then%
\begin{equation*}
\left( X_{1}+X_{2}\right) \left( h\right) \zeta _{3}=0
\end{equation*}%
Since $\zeta _{3}$ does not vanish, $\left( X_{1}+X_{2}\right) \left(
h\right) =0$ for any vector field $X_{1}+X_{2}$ and so $h$ is constant. Now,
Equations \ref{b1} and \ref{b2} become%
\begin{eqnarray*}
\bar{\mu}_{1}X_{1}-fg_{2}\left( X_{2},\zeta _{2}\right) \mathrm{grad}^{1}f
&=&0 \\
\bar{\mu}_{2}X_{2}+X_{1}\left( f\right) \zeta _{2}+\zeta _{1}\left( f\right)
X_{2} &=&0
\end{eqnarray*}%
Similarly, we can prove that $f$ is constant.
\end{proof}

The converse of the above result is considered in the following theorem.

\begin{theorem}
A vector field $\zeta =\zeta _{1}\in \mathfrak{X}\left( \left( M_{1}\times
_{f}M_{2}\right) \times _{h}M_{3}\right) $ is a concircular vector field if $%
\zeta _{1}$ is a concircular vector field with factor $\mu _{1}=\zeta
_{1}\left( \ln f\right) =\zeta _{1}\left( \ln h\right) $.
\end{theorem}

\section{Geometry of Sequential Warped Product Spacetimes}

We will state basic geometric formulas of two types sequential warped
product space-times, namely sequential generalized Robertson-Walker and
sequential standard static space-times. These results can be obtained by
direct applications of the results presented in Section 2.

\subsection{Sequential Generalized Robertson-Walker Space-times}

\begin{proposition}
\label{P2}Let $\bar{M}=\left( I\times _{f}M_{2}\right) \times _{h}M_{3}$ be
a sequential generalized Robertson-Walker space-time with metric $g=\left(
-dt^{2}\oplus f^{2}g_{2}\right) \oplus h^{2}g_{3}$ and also let $%
X_{i},Y_{i}\in \mathfrak{X}(M_{i})$ for any $i=2,3$. Then

\begin{enumerate}
\item $\bar{\nabla}_{\partial _{t}}\partial _{t}=0$

\item $\bar{\nabla}_{\partial _{t}}X_{i}=\bar{\nabla}_{X_{i}}\partial _{t}=%
\frac{\dot{f}}{f}X_{i},i=2,3$

\item $\bar{\nabla}_{X_{2}}Y_{2}=\nabla _{X_{2}}^{2}Y_{2}-f\dot{f}%
g_{2}\left( X_{2},Y_{2}\right) \partial _{t}$

\item $\bar{\nabla}_{X_{2}}X_{3}=\bar{\nabla}_{X_{3}}X_{2}=X_{2}\left( \ln
h\right) X_{3}$

\item $\bar{\nabla}_{X_{3}}Y_{3}=\nabla _{X_{3}}^{3}Y_{3}-hg_{3}\left(
X_{3},Y_{3}\right) \mathrm{grad}h$
\end{enumerate}
\end{proposition}

\begin{proposition}
Let $\bar{M}=\left( I\times _{f}M_{2}\right) \times _{h}M_{3}$ be a
sequential generalized Robertson-Walker spacetime with metric $\bar{g}%
=\left( -dt^{2}\oplus f^{2}g_{2}\right) \oplus h^{2}g_{3}$ and also let $%
X_{i},Y_{i},Z_{i}$ $\in \mathfrak{X}(M_{i})$. Then

\begin{enumerate}
\item \textrm{\={R}}$\left( \partial _{t},\partial _{t}\right) \partial
_{t}= $\textrm{\={R}}$\left( \partial _{t},\partial _{t}\right) Z_{j}=$%
\textrm{\={R}}$\left( X_{i},Y_{i}\right) Z_{j}=$\textrm{\={R}}$\left(
\partial _{t},Y_{2}\right) Z_{3}=0,i\neq j$

\item \textrm{\={R}}$\left( X_{2},Y_{2}\right) Z_{2}=R^{2}\left(
X_{2},Y_{2}\right) Z_{2}+\dot{f}^{2}\left[ g_{2}\left( X_{2},Y_{2}\right)
Y_{2}-g_{2}\left( Z_{2},Y_{2}\right) X_{2}\right] $,

\item \textrm{\={R}}$\left( \partial _{t},Y_{2}\right) \partial _{t}=\frac{%
\ddot{f}}{f}Y_{2}$,

\item \textrm{\={R}}$\left( \partial _{t},Y_{3}\right) \partial _{t}=\dfrac{1%
}{\bar{f}}\frac{\partial ^{2}h}{\partial t^{2}}Y_{3}$, $i,j=1,2$

\item \textrm{\={R}}$\left( \partial _{t},Y_{2}\right) Z_{2}=f\ddot{f}%
g_{2}\left( Y_{2},Z_{2}\right) \partial _{t}$

\item \textrm{\={R}}$\left( X_{2},Y_{3}\right) Z_{2}=\dfrac{-1}{h}%
H^{h}\left( X_{2},Z_{2}\right) Y_{3}$,

\item \textrm{\={R}}$\left( \partial _{t},Y_{3}\right) Z_{3}=hg_{3}\left(
Y_{3},Z_{3}\right) \nabla _{\partial _{t}}\mathrm{grad}h$,

\item \textrm{\={R}}$\left( X_{2},Y_{3}\right) Z_{3}=hg_{3}\left(
Y_{3},Z_{3}\right) \nabla _{X_{2}}\mathrm{grad}h$

\item \textrm{\={R}}$\left( X_{3},Y_{3}\right) Z_{3}=R^{3}\left(
X_{3},Y_{3}\right) Z_{3}-\left\Vert \mathrm{grad}h\right\Vert ^{2}\left[
g_{3}\left( X_{3},Y_{3}\right) Y_{3}-g_{3}\left( Z_{3},Y_{3}\right) X_{3}%
\right] $
\end{enumerate}
\end{proposition}

Now we consider the Ricci curvature \textrm{\={R}ic} of a sequential
generalized Robertson-Walker spacetime of the form $\bar{M}=\left( I\times
_{f}M_{2}\right) \times _{h}M_{3}.$

\begin{proposition}
Let $\bar{M}=\left( I\times _{f}M_{2}\right) \times _{h}M_{3}$ be a
sequential GRW spacetime with metric $\bar{g}=\left( -dt^{2}\oplus
f^{2}g_{2}\right) \oplus h^{2}g_{3}$ and also let $X_{i},Y_{i},Z_{i}$ $\in 
\mathfrak{X}(M_{i})$. Then

\begin{enumerate}
\item \textrm{\={R}ic}$\left( \partial _{t},\partial _{t}\right) =\dfrac{%
n_{2}}{f}\ddot{f}+\dfrac{n_{3}}{h}\frac{\partial ^{2}h}{\partial t^{2}}$

\item \textrm{\={R}ic}$\left( X_{2},Y_{2}\right) =$\textrm{Ric}$^{2}\left(
X_{2},Y_{2}\right) -g_{2}\left( X_{2},Y_{2}\right) f^{\sharp }-\dfrac{n_{3}}{%
h}H^{h}\left( X_{2},Y_{2}\right) $

\item \textrm{\={R}ic}$\left( X_{3},Y_{3}\right) =$\textrm{Ric}$^{3}\left(
X_{3},Y_{3}\right) -g_{3}\left( X_{3},Y_{3}\right) h^{\sharp }$

\item \textrm{\={R}ic}$\left( X_{i},Y_{j}\right) =0,i\neq j$
\end{enumerate}

where $f^{\sharp }=-f\ddot{f}-\left( n_{2}-1\right) \dot{f}^{2}$ and $%
h^{\sharp }=h\Delta h+\left( n_{3}-1\right) \left\Vert \mathrm{grad}%
h\right\Vert ^{2}$
\end{proposition}

A sequential GRW space-time $\bar{M}=\left( I\times _{f}M_{2}\right) \times
_{h}M_{3}$ is Einstein if%
\begin{equation*}
\mathrm{\bar{R}ic}\left( X,Y\right) =\mu \bar{g}\left( X,Y\right)
\end{equation*}%
We have the following cases. The first case is%
\begin{eqnarray*}
\mathrm{\bar{R}ic}\left( \partial _{t},\partial _{t}\right) &=&\mu \bar{g}%
\left( \partial _{t},\partial _{t}\right) \\
\dfrac{n_{2}}{f}\ddot{f}+\dfrac{n_{3}}{h}\frac{\partial ^{2}h}{\partial t^{2}%
} &=&-\mu
\end{eqnarray*}%
and the second case is%
\begin{equation*}
\mathrm{Ric}^{2}\left( X_{2},Y_{2}\right) -g_{2}\left( X_{2},Y_{2}\right)
f^{\sharp }-\dfrac{n_{3}}{h}H^{h}\left( X_{2},Y_{2}\right) =\mu
f^{2}g_{2}\left( X_{2},Y_{2}\right)
\end{equation*}%
and so%
\begin{equation*}
\mathrm{Ric}^{2}\left( X_{2},Y_{2}\right) =\dfrac{n_{3}}{h}H^{h}\left(
X_{2},Y_{2}\right) +\left( \mu f^{2}+f^{\sharp }\right) g_{2}\left(
X_{2},Y_{2}\right)
\end{equation*}%
and finally we have%
\begin{equation*}
\mathrm{Ric}^{3}\left( X_{3},Y_{3}\right) =\left( \mu h^{2}+h^{\sharp
}\right) g_{3}\left( X_{3},Y_{3}\right)
\end{equation*}

\begin{theorem}
Let $\bar{M}=\left( I\times _{f}M_{2}\right) \times _{h}M_{3}$ be an
Einstein sequential GRW space-time with metric $\bar{g}=\left( -dt^{2}\oplus
f^{2}g_{2}\right) \oplus h^{2}g_{3}$. Then,

\begin{enumerate}
\item $\mu =-\left( \dfrac{n_{2}}{f}\ddot{f}+\dfrac{n_{3}}{h}\frac{\partial
^{2}h}{\partial t^{2}}\right) $

\item $\left( M_{2},g_{2}\right) $ is Einstein with factor $\left( \mu
f^{2}+f^{\sharp }\right) $ if $H^{h}\left( X_{2},Y_{2}\right) =0$ for any $%
X_{2},Y_{2}\in \mathfrak{X}(M_{2})$ and

\item $\left( M_{3},g_{3}\right) $ is Einstein with factor $\left( \mu
h^{2}+h^{\sharp }\right) $.
\end{enumerate}
\end{theorem}

\begin{corollary}
Let $\bar{M}=\left( I\times _{f}M_{2}\right) \times _{h}M_{3}$ be an
Einstein sequential GRW space-time with metric $\bar{g}=\left( -dt^{2}\oplus
f^{2}g_{2}\right) \oplus h^{2}g_{3}$ and factor $\mu $. Then

\begin{enumerate}
\item $\left( \bar{M},\bar{g}\right) $ is Ricci flat if $n_{2}h\ddot{f}%
+n_{3}f\frac{\partial ^{2}h}{\partial t^{2}}=0$,

\item $\left( M_{2},g_{2}\right) $ is Ricci flat if $\mu f^{2}+f^{\sharp }=0$
and $H^{h}\left( X_{2},Y_{2}\right) =0$ for any $X_{2},Y_{2}\in \mathfrak{X}%
(M_{2})$ and

\item $\left( M_{3},g_{3}\right) $ is Ricci flat if $\mu h^{2}+h^{\sharp }=0$%
.
\end{enumerate}
\end{corollary}

The converse of the above theorem is considered in the following result.

\begin{theorem}
Let $\bar{M}=\left( I\times _{f}M_{2}\right) \times _{h}M_{3}$ be a
sequential GRW space-time with metric $\bar{g}=\left( -dt^{2}\oplus
f^{2}g_{2}\right) \oplus h^{2}g_{3}$. Then $\left( \bar{M},\bar{g}\right) $
is Einstein with factor $\mu $ if

\begin{enumerate}
\item $H^{h}\left( X_{2},Y_{2}\right) =0$ for any $X_{2},Y_{2}\in \mathfrak{X%
}(M_{2})$,

\item $\left( M_{i},g_{i}\right) $ is Einstein with factor $\mu _{i},i=2,3$,

\item $\mu _{2}+f\ddot{f}+\left( n_{2}-1\right) \dot{f}^{2}=\mu f^{2}$

\item $\mu _{3}+h\dfrac{\partial ^{2}h}{\partial t^{2}}-\left(
n_{3}-1\right) \left\Vert \mathrm{grad}h\right\Vert ^{2}=\mu h^{2}$

\item $\dfrac{n_{2}}{f}\ddot{f}+\dfrac{n_{3}}{h}\dfrac{\partial ^{2}h}{%
\partial t^{2}}=-\mu $
\end{enumerate}
\end{theorem}

\subsection{Sequential Standard Static Space-times}

\begin{theorem}
\label{P3}Let $\bar{M}=\left( M_{1}\times _{f}M_{2}\right) \times _{h}I$ be
a sequential standard static space-time with metric $\bar g=\left(
g_{1}\oplus f^{2}g_{2}\right) \oplus h^{2}\left( -dt^{2}\right) $ and also
let $X_{i},Y_{i}$ $\in \mathfrak{X}(M_{i})$. Then

\begin{enumerate}
\item $\bar{\nabla}_{X_{1}}Y_{1}=\nabla _{X_{1}}^{1}Y_{1}$

\item $\bar{\nabla}_{X_{1}}X_{2}=\bar{\nabla}_{X_{2}}X_{1}=X_{1}\left( \ln
f\right) X_{2}$

\item $\bar{\nabla}_{X_{2}}Y_{2}=\nabla _{X_{2}}^{2}Y_{2}-fg_{2}\left(
X_{2},Y_{2}\right) \mathrm{grad}^{1}f$

\item $\bar{\nabla}_{X_{i}}\partial _{t}=\bar{\nabla}_{\partial
_{t}}X_{i}=X_{i}\left( \ln h\right) \partial _{t},i=1,2$

\item $\bar{\nabla}_{\partial _{t}}\partial _{t}=h\mathrm{grad}h$
\end{enumerate}
\end{theorem}

\begin{theorem}
Let $\bar{M}=\left( M_{1}\times _{f}M_{2}\right) \times _{h}I$ be a
sequential standard static space-time with metric $\bar g=\left( g_{1}\oplus
f^{2}g_{2}\right) \oplus h^{2}\left( -dt^{2}\right) $ and also let $%
X_{i},Y_{i},Z_{i}$ $\in \mathfrak{X}(M_{i})$. Then

\begin{enumerate}
\item \textrm{\={R}}$\left( X_{1},Y_{1}\right) Z_{1}=R^{1}\left(
X_{1},Y_{1}\right) Z_{1}$

\item \textrm{\={R}}$\left( X_{2},Y_{2}\right) Z_{2}=R^{2}\left(
X_{2},Y_{2}\right) Z_{2}-\left\Vert \mathrm{grad}^{1}f\right\Vert ^{2}\left[
g_{2}\left( X_{2},Y_{2}\right) Y_{2}-g_{2}\left( Z_{2},Y_{2}\right) X_{2}%
\right] $

\item \textrm{\={R}}$\left( X_{1},Y_{2}\right) Z_{1}=\dfrac{-1}{f}%
H_{1}^{f}\left( X_{1},Z_{1}\right) Y_{2}$

\item \textrm{\={R}}$\left( X_{1},Y_{2}\right) Z_{2}=fg_{2}\left(
Y_{2},Z_{2}\right) \nabla _{X_{1}}^{1}\mathrm{grad}^{1}f$

\item \textrm{\={R}}$\left( X_{1},Y_{2}\right) \partial _{t}=$\textrm{\={R}}$%
\left( \partial _{t},\partial _{t}\right) \partial _{t}=$\textrm{\={R}}$%
\left( X_{i},Y_{i}\right) Z_{j}=0,i\neq j$

\item \textrm{\={R}}$\left( X_{i},\partial _{t}\right) Z_{j}=\dfrac{-1}{h}%
H^{h}\left( X_{i},Z_{j}\right) \partial _{t},i,j=1,2$

\item \textrm{\={R}}$\left( X_{i},\partial _{t}\right) \partial
_{t}=-h\nabla _{X_{i}}\mathrm{grad}h,i=1,2$
\end{enumerate}
\end{theorem}

Now consider the Ricci curvature \textrm{\={R}ic} of a sequential standard
static space-time of the form $\left( M_{1}\times _{f}M_{2}\right) \times
_{h}I.$

\begin{theorem}
Let $\bar{M}=\left( M_{1}\times _{f}M_{2}\right) \times _{h}I$ be a
sequential standard static space-time with metric $\bar g=\left( g_{1}\oplus
f^{2}g_{2}\right) \oplus h^{2}\left( -dt^{2}\right)$ and also let $%
X_{i},Y_{i},Z_{i}$ $\in \mathfrak{X}(M_{i})$. Then

\begin{enumerate}
\item \textrm{\={R}ic}$\left( X_{1},Y_{1}\right) =$\textrm{Ric}$^{1}\left(
X_{1},Y_{1}\right) -\dfrac{n_{2}}{f}H_{1}^{f}\left( X_{1},Y_{1}\right) -%
\dfrac{1}{h}H^{h}\left( X_{1},Y_{1}\right) $

\item \textrm{\={R}ic}$\left( X_{2},Y_{2}\right) =$\textrm{Ric}$^{2}\left(
X_{2},Y_{2}\right) -g_{2}\left( X_{2},Y_{2}\right) f^{\sharp }-\dfrac{1}{h}%
H^{h}\left( X_{2},Y_{2}\right) $

\item \textrm{\={R}ic}$\left( \partial _{t},\partial _{t}\right) =h\Delta h$

\item \textrm{\={R}ic}$\left( X_{i},Y_{j}\right) =0,i\neq j$

where $f^{\sharp }=f\Delta ^{1}f+\left( n_{2}-1\right) \left\Vert \mathrm{%
grad}^{1}f\right\Vert ^{2}$.
\end{enumerate}
\end{theorem}

A sequential standard static space-time $\left( M_{1}\times _{f}M_{2}\right)
\times _{h}I$ is Einstein with factor $\mu $ if%
\begin{equation}
\mathrm{\bar{R}ic}\left( X,Y\right) =\mu \bar{g}\left( X,Y\right)
\label{e41}
\end{equation}%
In this case%
\begin{equation*}
\mu =-\frac{\Delta h}{h}
\end{equation*}%
But taking the trace of equation (\ref{e41}) we get that%
\begin{equation*}
\mu =\frac{r}{n_{1}+n_{2}+1}
\end{equation*}%
where $r$ is the scalar curvature i.e.%
\begin{equation*}
r=-\frac{\Delta h}{h}\left( n_{1}+n_{2}+1\right)
\end{equation*}%
Moreover,%
\begin{equation*}
\mathrm{Ric}^{1}\left( X_{1},Y_{1}\right) -\dfrac{n_{2}}{f}H_{1}^{f}\left(
X_{1},Y_{1}\right) -\dfrac{1}{h}H^{h}\left( X_{1},Y_{1}\right) =\mu
g_{1}\left( X_{1},Y_{1}\right)
\end{equation*}%
and%
\begin{equation*}
\mathrm{Ric}^{2}\left( X_{2},Y_{2}\right) -g_{2}\left( X_{2},Y_{2}\right)
f^{\sharp }-\dfrac{1}{h}H^{h}\left( X_{2},Y_{2}\right) =\mu f^{2}g_{2}\left(
X_{2},Y_{2}\right)
\end{equation*}

\begin{corollary}
Let $\bar{M}=\left( M_{1}\times _{f}M_{2}\right) \times _{h}I$ be an
Einstein sequential standard static space-time with metric $\bar{g}=\left(
g_{1}\oplus f^{2}g_{2}\right) \oplus h^{2}\left( -dt^{2}\right) $. Then the
scalar curvature $r$ of $\bar{M}$ is given by%
\begin{equation*}
r=-\frac{\Delta h}{h}\left( n_{1}+n_{2}+1\right)
\end{equation*}
\end{corollary}

\begin{corollary}
Let $\bar{M}=\left( M_{1}\times _{f}M_{2}\right) \times _{h}I$ be an
Einstein sequential standard static space-time with metric $\bar{g}=\left(
g_{1}\oplus f^{2}g_{2}\right) \oplus h^{2}\left( -dt^{2}\right) $. Then

\begin{enumerate}
\item $\left( M_{1},g_{1}\right) $ is Einstein with factor $\mu $ if $%
n_{2}hH_{1}^{f}\left( X_{1},Y_{1}\right) -fH^{h}\left( X_{1},Y_{1}\right) =0$%
,

\item $\left( M_{2},g_{2}\right) $ is Einstein with factor $\mu
f^{2}+f^{\sharp}$ if $H^{h}\left( X_{2},Y_{2}\right) =0$
\end{enumerate}
\end{corollary}


\begin{thebibliography}{99}
\bibitem{Apostolopoulos:2005} P. S. Apostolopoulos and J. G. Carot, \emph{%
Conformal symmetries in warped manifolds}, Journal of Physics: Conference
Series \textbf{8} (2005), 28--33.

\bibitem{AD1} D.E. Allison , \emph{Energy conditions in standard static
space-times}, General Relativity and Gravitation, \textbf{20}(1998), No. 2,
115-122.

\bibitem{AD} D.E. Allison, \emph{Geodesic Completeness in Static Space-times}%
, Geometriae Dedicata, \textbf{26} (1988), 85-97.

\bibitem{GES} D.E. Allison and B. \"{U}nal, \emph{Geodesic Structure of
Standard Static Space-times}, Journal of Geometry and Physics, \textbf{46}%
(2003), No.2, 193-200.

\bibitem{ACD} M.T. Anderson, P. T. Chrusciel and E. Delay, \emph{%
Non-trivial, static, geodesically complete, vacuum space-times with a
negative cosmological constant}, J. High Energy Phys. \textbf{10}(2002).

\bibitem{BEE} J. K. Beem, P. E. Ehrlich and K. L. Easley, \textit{Global
Lorentzian Geometry}, (2nd Ed.), Marcel Dekker, New York, 1996.

\bibitem{Berestovskii:2008} V. N. Berestovskii, Yu. G. Nikonorov, \emph{%
Killing vector fields of constant length on Riemannian manifolds}, Siberian
Mathematical Journal, \textbf{49}(2008), Issue 3 , pp 395-407

\bibitem{Besse2008} A. L. Besse\textbf{, }\emph{Einstein Manifolds},
Classics in Mathematics, Springer-Verlag, Berlin, 2008.

\bibitem{Bishop:1969} R. L. Bishop and B. O'Neill, \emph{Manifolds of
negative curvature}, Trans. Amer. Math. Soc. \textbf{145} (1969), 1-49.

\bibitem{Chen2014} B. Y. Chen, \emph{A simple characterization of
generalized Robertson-Walker space-times}, General Relativity and
Gravitation, \textbf{46}(2014), 1833-1839

\bibitem{Dobarro:1987} F. Dobarro E. Lam\'{\i} Dozo, \emph{Scalar Curvature
and Warped Products Of Riemann Manifolds, }Trans. Amer. Math. Soc. \textbf{%
303} (1987), no. 1, 161-168.

\bibitem{Unal:2012} F. Dobarro and B. \"{U}nal, \emph{Characterizing killing
vector fields of standard static space-times}, J. Geom. Phys. \textbf{62}
(2012), 1070--1087.

\bibitem{Dobbaro2005} F. Dobarro, B. \"{U}nal, \emph{Curvature of multiply
warped products}, Journal of Geometry and Physics, \textbf{51}(2005), no. 1,
75-106.

\bibitem{Dobbaro2008} F. Dobarro, B. \"Unal, \emph{Curvature in special base
conformal warped products}, Acta Appl. Math. \textbf{104}(2008), no. 1, 1-46.

\bibitem{Dumitru:2015} D. Dumitru, \emph{On multiply Einstein warped products%
}, Annals of the Alexandru Ioan Cuza University - Mathematics, to appear.

\bibitem{Duggala:2005} K. L. Duggal and R. Sharma,\emph{Conformal killing
vector fields on spacetime solutions of Einstein's equations and initial data%
}, Nonlinear Analysis \textbf{63} (2005) e447 -- e454.

\bibitem{Ivancevic:2007} V. G. Ivancevic and T. T. Ivancevic, \emph{Applied
Differential Geometry: A Modern Introduction}, World Scientific Publishing
Co. Ltd, London, 2007.

\bibitem{Kuhnel:1997} W. Kuhnel and H. Rademacher, \emph{Conformal vector
fields on pseudo-Riemannian spaces}, Journal of Geometry and its
Applications, \textbf{7}(1997), 237--250.

\bibitem{Operea:2008} T. Oprea, \emph{2-Killing Vector Fields on Riemannian
Manifolds}, Balkan Journal of Geometry and Its Applications, \textbf{13}%
(2008), No.1, 87-92.

\bibitem{Oneill:1983} B. O'Neill, \emph{Semi-Riemannian Geometry with
Applications to Relativity}, Academic Press Limited, London, 1983.

\bibitem{Sanchez99} M. S\'{a}nchez, \emph{On the Geometry of Generalized
Robertson-Walker Spacetimes: Curvature and Killing fields}, J. Geom. Phys., 
\textbf{31} (1999), no.1, 1-15.

\bibitem{Sanchez98} M. S\'{a}nchez, \emph{\ On the Geometry of Generalized
Robertson-Walker Spacetimes: geodesics}, Gen. Relativ. Gravitation, \textbf{%
30} (1998), no.6, 915-932.

\bibitem{Shenawy:2015} S. Shenawy and B. \"Unal, \emph{$2-$Killing vector
fields on warped product manifolds}, International Journal of Mathematics, 
\textbf{26}(2015), 17 pages.

\bibitem{Shenawy:2016} S. Shenawy and B. \"Unal, \emph{The $W_2-$curvature
tensor on warped product manifolds and applications}, International Journal
of Geometric Methods in Modern Physics, \textbf{13}(2016), No. 07, 1650099
(14 pages).

\bibitem{Stephani:2003} H. Stephani, D. Kramer, M. MacCallum, C.
Hoenselaers, and E. Herlt, \emph{Exact Solutions of Einstein's Field
Equations.} Second Edition, Cambridge University Press, Cambridge, 2003.

\bibitem{Unal2001A} B. \"Unal, \emph{Multiply warped products}, Journal of
Geometry and Physics, \textbf{34}(2001), no. 3-4, 287-301.

\bibitem{Unal:2001} B. \"Unal, \emph{Doubly warped products}, Differential
Geometry and its Applications \textbf{15} (2001), 253--263.
\end{thebibliography}
\end{document}